\documentclass{article}
\usepackage{amsmath,amssymb,amsthm}
\usepackage{graphics}

\newtheorem{theorem}{Theorem}
\newtheorem{lemma}{Lemma}

\newcommand{\orb}{\mathop{\rm Orb}\nolimits}

\newcommand{\gl}{\mathop{\rm GL}\nolimits}

\newcommand{\md}{\mathop{\rm (mod}\nolimits}

\newcommand{\cat}{\mathop{\rm cat}\nolimits}
\newtheorem{cor}{Corollary}

\title {Borsuk-Ulam type theorems for manifolds}

\author {Oleg R. Musin \thanks{Research supported in part by NSF grant DMS-0807640 and NSA grant MSPF-08G-201.}}

\begin{document}
\date{}
\maketitle

\begin{abstract} This paper establishes a Borsuk-Ulam type theorem for  PL-manifolds with a finite group  action, depending on the free equivariant cobordism class of a manifold. In particular, necessary and sufficient conditions are considered for a manifold with a free involution to be a Borsuk-Ulam type.
\end{abstract}

\medskip

\noindent {\bf Keywords:} Borsuk - Ulam theorem, group action, equivariant cobordism.

\section{Introduction}

The Borsuk-Ulam theorem states that any continuous antipodal $f:{\Bbb S}^n \to {\Bbb R}^n$  has zeros. (A map $f$ is called {\it antipodal} if $f(-x)=-f(x)$.)
One of the most interesting proofs of this theorem is B\'ar\'any's geometric proof \cite{Bar} (see also \cite[Sec. 2.2]{Mat}). A very similar proof was given in  \cite{MW}, and \cite{Ste} has several more references for proofs of this type.

Let $X={\Bbb S}^n\times[0,1], \, X_0={\Bbb S}^n\times \{0\},$ and  $X_1={\Bbb S}^n\times \{1\}.$
Let $\tau(x,t)=(-x,t)$, where $(x,t)\in X$,  $x\in {\Bbb S}^n$, and $t\in[0,1]$. Clearly, $\tau$ is a free involution on $X$.

The first step of B\'ar\'any's proof is to show that any continuous antipodal (i.e. $F(\tau(x))=-F(x)$) map $F: X\to {\Bbb R}^n$ can be approximated by ``sufficiently generic'' antipodal  maps (see  \cite[Sec. 2.2]{Mat}). 

Let $f_i:{\Bbb S}^n \to {\Bbb R}^n$, where $i=0,1$, be antipodal generic maps.
Let $F(x,t)=tf_1(x)+(1-t)f_0(x)$. 
Since $F$ is generic, the set $Z_F:=F^{-1}(0)$ is a manifold of dimension one. Then $Z_F$ consists of arcs $\{\gamma_k\}$ with ends in $Z_{f_i}:=Z_F\bigcap X_i=f_i^{-1}(0)$ and cycles which do not intersect $X_i$. Note that $\tau(Z_F)=Z_F$ and $\tau(\gamma_i)=\gamma_j$ with $i\ne j$.  Therefore, $(Z_F,Z_{f_0},Z_{f_1})$ is a ${\Bbb Z}_2$-{\it cobordism}.
It is not hard to see that $Z_{f_0}$ {\it is ${\Bbb Z}_2$-cobordant to} $Z_{f_1}$  if and only if $|Z_{f_1}|=|Z_{f_0}|\md 4).$

To complete the proof, take $f_0$ as the standard orthogonal projection of  ${\Bbb S}^n$ onto ${\Bbb R}^n$:
 $$
 f_0(x_1,\ldots,x_n,x_{n+1})=(x_1,\ldots,x_n), \; \mbox{ where } \; x_1^2+\ldots+x_{n+1}^2=1.
 $$
 Since $|Z_{f_0}|=2$, we have $|Z_{f_1}|=2\md 4).$ This equality shows that for any  antipodal generic $f_1$ the set $Z_{f_1}=f_1^{-1}(0)$ is not empty.

\medskip

Our analysis of this proof shows that it can be extended for a wide class of manifolds. For instance,  consider two-dimensional   orientable manifolds $N=M^2_g$ of even genus $g$ and non-orientable manifolds $N=P^2_m$ with even $m$. Without loss of generality,  we can assume that $N$ is  ``centrally symmetric" embedded to ${\Bbb R}^k$, where $k=3$ for $N=M^2_g$ and $k=4$ for $N=P^2_m$. That means $A(N)=N$, where $A(x)=-x$ for $x\in{\Bbb R}^k$. Then $T:=A|_N:N\to N$ is a free involution. It can be shown that there is a projection of $N\subset{\Bbb R}^k$ into a 2-plane $R$ passing through the origin 0 with  $|Z_{f_0}|=2.$ (See details in Corollary \ref{corMcM}.)

Actually, B\'ar\'any's proof can be (almost word for word)  applied for $N$. Namely, the following statement holds:\\
{\it Let $N=M^2_g$ with even $g$ or $N=P^2_m$ with even $m$. Then for any continuous $h:N\to {\Bbb R}^2$ there is $x\in N$ such that $h(T(x))=h(x)$.}

 Note that $N$ can be represented as a connected sum $M\# M$. This statement can be extended for $N=M\# M$ with any closed manifold $M$ (see Corollary \ref{corMcM}).\footnote{It is not hard to prove the corollary using the standard techniques developed in \cite{CF60}. However, we couldn't find this kind statements in the Borsuk-Ulam theorem literature.}  We see that one of the most important steps here is the existence of a generic equivariant $f_0:N^n\to {\Bbb R}^n$ with $|Z_{f_0}|=2\md 4).$

This approach works for any free group action. Namely,
for both technical steps a ``generic approximation lemma'' and a ``$(Z_F,Z_{f_0},Z_{f_1})$ cobordism lemma'' can be extended for free  actions of a finite group $G$ (Section 2).

Let $G$ be a finite group acting free on a closed connected PL-manifold $M^m$ and linearly on ${\Bbb R}^n$.  In Section 3 we show that a Borsuk-Ulam type theorem for $M$, depending on the free equivariant cobordism class of $M$ (Theorem 1). For the case $m=n$,  Theorem \ref{thmdeg} shows that  if  there is  a  continuous equivariant  transversal to zeros  $h: M^n\to {\Bbb R}^n$  with $|Z_{h}|=|G|\md{2|G|})$, then for any continuous equivariant  $f: M^n\to {\Bbb R}^n$ the  zero set $Z_f$ is not empty.

In Section 4, this approach is applied to the classical case: manifolds with free involutions. There  we give  necessary and sufficient
conditions for ${\Bbb Z}_{2}$-manifolds to have a Borsuk-Ulam type theorem (Theorem \ref{thmBUT}).

Lemma \ref{lm4} shows that the zero set $Z_f$ is invariant up to $G$-cobordisms. Based on this fact, in Section 5 we define homomorphisms $\mu$ of free $G$-cobordism groups which can be considered as obstructions for $G$-maps. Theorem \ref{thmGcb}  shows that for a free $G$-manifold $M^m$ and a given linear action $G$ on ${\Bbb R}^n$, if this invariant is not zero in the free equivariant cobordisms, then for any continuous equivariant  $f: M^m\to {\Bbb R}^n$ the  set $Z_f$ is not empty. In this section also we also provide some applications of Theorem \ref{thmGcb}  for the case $G=({\Bbb Z}_2)^k$.

The main goal of this paper is to show that the geometric proof gives a method for checking whether  a $G$-manifold $M$ is of the Borsuk-Ulam type. Namely, from Theorem \ref{thmdeg} and its extension Corollary \ref{cor2}, it follows that if there exists a generic equivariant $h: M^m\to {\Bbb R}^n$ with $[Z_h]\ne0$  in the correspondent group of cobordisms, then for any continuous equivariant  $f: M^m\to {\Bbb R}^n$ the   set of zeros $Z_f\ne\emptyset$. 

\section{Generic $G$-maps}

If $G$ is a group and $X$ is a set, then a  group action of $G$ on $X$ is a binary function $G\times X\to X$
denoted $(g,x)\to g(x)$ which satisfies the following two axioms: (1) $(gh)(x) = g(h(x))$ for all $g, h$ in $G$ and $x$ in $X$; (2) $e(x) = x$ for every $x$ in $X$ (where $e$ denotes the identity element of $G$).
The set $X$ is called a $G$-set. The group $G$ is said to act on $X$.

Let a finite group $G$ act on a set $X$. We say that $Y\subset X$ is a $G$-subset if $g(y)\in Y$ for all $g\in G$ and $y\in Y$, i.e. $G(Y)=Y$. Clearly, for any   $x\in X$ the orbit $G(x):=\{g(x),\, g\in G\}$ is a $G$-subset.
Denote by $\orb(X,G)$ the set of orbits of $G$ on $X$, i.e. $\orb(X,G)=\{G(x)\}$

Let $G_x:=\{g\in G|g(x)=x\}$ denote the {\it stabilizer} (or {\it isotropy subgroup}). Let $X^H=\{x\in X|g(x)=x, \forall g\in H\}$ denote the fixed point set of a subgroup $H\subset G$. If $H$ is isomorphic to $G_x\ne e$, then $X^H\ne\emptyset$.

Recall that a group action is called {\it free} if $G_x=\{e\}$ for all $x$, i.e. $g(x)=x$ if only if $g=e$. Note that for a free action $G$ on $X$ each orbit $G(x)$ consists of $|G|$ points.

Here we consider {\it piece-wise linear} (or {\it PL}, or {\it simplicial}) $G$-manifolds $X$, i.e.  there is a triangulation $\Lambda$ of $X$ such that for any $g\in G$, any simplex $\sigma$ of  $\Lambda$ is mapped bijectively onto the simplex $g(\sigma)$. The triangulation $\Lambda$ is called {\it equivariant}. Actually, any smooth $G$-manifold admits an equivariant triangulation \cite{I}.

Every free action of a finite group $G$  on a compact PL-manifold $X$ admits an equivariant triangulation such that for each simplex  there are no vertices in the same orbit. Otherwise, we indeed can subdivide these simplices.

A map $f:X\to Y$ of  $G$-sets $X$ and $Y$ is called {\it equivariant} (or {\it $G$-map}) if  $f(g(x))=g(f(x))$ for all $g\in G, x\in X$.


For a PL $G$-space (in particular, for a PL  $G$-manifold) we say that an equivariant map $f:X\to{\Bbb R}^n$ is {\it simplicial} if $f$ is a linear map for each simplex $\sigma\in\Lambda$. For an equivariant triangulation $\Lambda$, any simplicial  map $f:X \to {\Bbb R}^n$ is uniquely determined by the set of vertices $V(\Lambda)$. Indeed, for each simplex $\sigma$ of $\Lambda$, $f$ is linear, and therefore is determined by the vertices of $\sigma$.

Any equivariant continuous map can be approximated by an equivariant simplicial map.

\begin{lemma} 
\label{lemma:lemma1}	Let $G$ be a finite group  acting linearly on ${\Bbb R}^n$. Let $X$ be a compact simplicial  $G$-space.    Then for any equivariant  continuous   $f:X \to{\Bbb R}^n$ and $\varepsilon>0$,  there is  an equivariant  simplicial  map $\bar f:X \to {\Bbb R}^n$ such that
$\; ||f(x)-\bar f(x)||\le \varepsilon$  for all $x\in X$.
\end{lemma}
\begin{proof}
It is well known (see, for instance, \cite[Section 2.2]{FR})
that any continuous map on a compact PL-complex can be approximated by simplicial maps. Therefore there exists a triangulation $\Lambda$ and a simplicial  map $\hat f$ such that $||{f(x)-\hat f(x)}||\le \varepsilon$ for all $x\in X$. Clearly, there exists an equivariant subdivision $\Lambda'$ of $\Lambda$ with  $G(V(\Lambda))\subset V(\Lambda')$. For $v\in V(\Lambda')$, let $\bar f(v):=f(v)$. Then $\bar f|_{V(\Lambda')}$ defines a simplicial map with  $||f(x)-\bar f(x)||\le \varepsilon$.
\end{proof}

Let $\rho:G\to \gl(n,{\Bbb R})$ be a representation of a group $G$ on ${\Bbb R}^n$. In other words, $G$ is  acting linearly on ${\Bbb R}^n$. Then $F:={(\Bbb R}^n)^G$ is a linear subspace of ${\Bbb R}^n$. Denote by $L$ an invariant linear subspace of ${\Bbb R}^n$ which is transversal to $F$. Then $L^G=\{0\}$. So without loss of generality we can consider only linear actions with  $({\Bbb R}^n)^G=\{0\}$, i.e. linear actions on ${\Bbb R}^n$ such that the fixed point set of these actions is the origin.

Let  $f:X\to{\Bbb R}^n$. Denote by   $Z_f$  the set of zeros, i.e. $$Z_f=\{x\in X: f(x)=0\}.$$

In this paper we need generic maps with respect to $Z_f$. Let $X$ be a simplicial $m$-dimensional $G$-manifold.  Set $O_\varepsilon:=\{v\in{\Bbb R}^n: ||v||<\varepsilon\}$. A continuous equivariant  $f:X^m\to{\Bbb R}^n$ is called {\it transversal to zeros} if $Z_f$ is a manifold of dimension $m-n$, and there is $\varepsilon>0$ such that for any $v\in O_\varepsilon$, the sets $f^{-1}(v)$ and $Z_f$ are homeomorphic.

\medskip

\noindent{\bf Definition.} Given a finite group $G$. Let $X^m$ be a PL free  $G$-manifold. Let $\Lambda$ be an equivariant triangulation of $X$.  We say that an equivariant simplicial map $f:X\to{\Bbb R}^n$ is {\it generic} (with respect to zeros) if $Z_f$ does not intersect the $(n-1)$-skeleton $\Lambda(n-1)$ of $\Lambda$. (Recall that the $k$-skeleton of $\Lambda$ is the subcomplex $\Lambda(k)$  that consists of all simplices of dimension at most $k$.)

\medskip

All simplicial generic maps are transversal to zeros.

\begin{lemma} 
\label{lm2}	Let $G$ be a finite group acting linearly on ${\Bbb R}^n$ with $({\Bbb R}^n)^G=\{0\}$.  Let $X^m$ be a PL compact $m$-dimensional free $G$-manifold with or without boundary.  Let $f:X\to{\Bbb R}^n$ be an equivariant simplicial generic map. If $n\le m$, then $Z_f$ is an invariant submanifold of $X$ of dimension $m-n$. Moreover, if the boundary $\partial Z_f$ is not empty, then it lies in $\partial X$.
\end{lemma}
\begin{proof}
Note that
$Z_f\subset X$ is a locally polyhedral surface  consisting of $(m-n)$-dimensional cells.  This is because for each $m$-simplex $\sigma$  we have the linear map $f|_\sigma:\sigma\to {\Bbb R}^n$, and $Z_f\cap\sigma$ is defined by the generic linear equation $f(x)=0$.  Hence, the components of $Z_f$ are PL manifolds.
\end{proof}

For $G$-spaces, the Tietze-Gleason theorem states: {\it Let a compact group $G$  act on $X$ with a closed invariant set $A$. Let $G$ also acts linearly on ${\Bbb R}^n$. Then any equivariant $f:A\to {\Bbb R}^n$ extends to $f:X\to {\Bbb R}^n$}  (see \cite[Theorem 2.3]{Br}).


It is known from results of Bierstone \cite{Bier} and Field \cite{Field} on $G$-transversality theory that the zero set is a stratified set. If $G$ acts free on a compact smooth $G$-manifold $X$, then this theory implies that the set of non-generic equivariant smooth $f:X^m\to {\Bbb R}^n$ has measure zero in the space of all smooth $G$-maps. Let us extend this result as well as the Tietze-Gleason theorem for the  PL case.

\begin{lemma} 
\label{lm3}	Let $G$ be a finite group acting linearly on ${\Bbb R}^n$ with $({\Bbb R}^n)^G=\{0\}$. Let $X^m$ be a PL compact free $G$-manifold  with or without boundary.   Let $\Lambda$ be an equivariant triangulation of $X$.  Let $A$ be an invariant subset  of $V(\Lambda)$ or let $A=\emptyset$. Let $h:A\to {\Bbb R}^n$ be a given equivariant generic map.
Then  the set of non-generic simplicial maps $f:X \to{\Bbb R}^n$ with $f|_A=h$ has measure zero in the space of all possible equivariant maps $f:V(\Lambda)\to {\Bbb R}^n$ with $f|_A=h$.
\end{lemma}
\begin{proof} Let $V':=V(\Lambda)\setminus A$. Let $C\in \orb(V',G)$ be an orbit of $G$ on $V'$. Then $|C|=p$, where $p:=|G|$. Since for any equivariant $f:V(\Lambda) \to{\Bbb R}^n$ and $v\in V(\Lambda)$  we have $f(g(v))=g(f(v))$, the vector $u=f(v)\in {\Bbb R}^n$ yields the map $f:C\to {\Bbb R}^n$.

Suppose $G$ has $k$ orbits on $V'$. Take in each orbit $C_i\in\orb(V',G)$ a vertex $x_i$. Denote by $M_G$ the space of all equivariant maps $f:V'\to {\Bbb R}^n$. Then the space $M_G$ is of  dimension $N=kn$.

Consider an $(n-1)$-simplex $\sigma\in\Lambda$ with vertices $v_1,\ldots,v_n$ which are not all from $A$. Note that $f\in M_G$ is not generic  on $\sigma$ if  $\, 0\in f(\sigma)$.
In particular, the hyperplane in ${\Bbb R}^n$ which is defined by vectors $f(v_1),\ldots,f(v_n)$  passes through the origin 0. In other words, the determinant of $n$ vectors $f(v_1),\ldots,f(v_n)$ equals zero.

This constraint gives a proper algebraic subvariety $s(\sigma)$ in $M_G$. Let  $v_1,\ldots,v_d$ be vertices of $\sigma$ which are not in $A$.   Then $s(\sigma)$ is of dimension $N-1$.  Indeed,  $u_i=f(v_i),$ $i=1,\ldots,d$, are vectors in ${\Bbb R}^{n}$. Then the equation $$\det(u_1,\ldots,u_d,h(v_{d+1}),\ldots,h(v_n))=0$$  defines the subvariety $s(\sigma)$ in $M_G$ of dimension $N-1$.

Since the union of all subvarieties $s(\sigma)$ with $\dim{\sigma}=n-1$ is at most of dimension $N-1$, it has measure zero in $M_G$.  From this it follows that the set of all non-generic $f$ has measure zero in this space.
\end{proof}

Let $M_0$ and $M_1$  be closed $m$-dimensional simplicial manifolds with free actions of a finite group  $G$.  Then an $(m+1)$-dimensional simplicial free $G$-manifold $W$ is called a {\it free $G$-cobordism} $(W,M_0,M_1)$ if the boundary $\partial W$ consists of $M_0$ and $M_1$, and the action $G$ on $W$ respects actions on $M_i$.

\begin{lemma} \label{lm4}
Let $G$ be a finite group acting linearly on ${\Bbb R}^n$ with $({\Bbb R}^n)^G=\{0\}$. Let $(W,M_0,M_1)$  be a free $G$-cobordism. Let $f_i: M_i^m\to{\Bbb R}^n$, $i=0,1$, be equivariant simplicial generic maps.   Then there is an equivariant simplicial generic  $F:W \to {\Bbb R}^n$ with $F|_{M_i}=f_i$, and  such that $(Z_F,Z_{f_0},Z_{f_1})$ is a free $G$-cobordism.
\end{lemma}
\begin{proof}  The existence of such a generic $F$ follows from Lemma \ref{lm3}.
Lemma \ref{lm2} implies that $Z_F$ is a $G$-submanifold of $W$.
 Clearly, $g(Z_F)=Z_F$ and $g(Z_{f_i})=Z_{f_i}$  for all $g\in G$. Since
$\partial Z_F=Z_{f_0}\bigsqcup Z_{f_0}$ is an $(m-n)$-dimensional $G$-manifold, we have a free $G$-cobordism $(Z_F,Z_{f_0},Z_{f_1})$.
\end{proof}

\section{$G$-maps and a Borsuk-Ulam type theorem}

In this section we consider a Borsuk-Ulam type theorem for the case $m=n$.

\medskip

\noindent{\bf Definition.} Given a finite group $G$ acting free  on a closed PL-manifold $M^n$  and acting linearly on ${\Bbb R}^n$ with $({\Bbb R}^n)^G=\{0\}$.
Let $f: M^n\to {\Bbb R}^n$ be a  continuous equivariant transversal to zeros map. Since  $Z_f$ is a finite free $G$-invariant subset of $M$, we have $|Z_f|=k\,|G|$, where $k=0,1,2,\ldots$ Set 
$\deg_G(f):=1$ if $k$ is odd, and $\deg_G(f):=0$ if $k$ is even.

\begin{lemma} 
\label{lm5}	Let $G$ be a finite group acting linearly on ${\Bbb R}^n$ with $({\Bbb R}^n)^G=\{0\}$. Let $(W^{n+1},M_0^n,M_1^n)$ be a free $G$-cobordism. Let $h_i: M_i^n\to {\Bbb R}^n, \; i=0,1,$  be  continuous equivariant transversal to zeros maps. Then $\deg_G(h_0)=\deg_G(h_1).$
\end{lemma}
\begin{proof}
Lemma \ref{lemma:lemma1} and Lemma \ref{lm3} yield that for any $h_i: M_i^n\to {\Bbb R}^n$ and  $\varepsilon>0$ there is a generic map $a_{i,\varepsilon}$ such that $||a_{i,\varepsilon}(x)-h_i(x)||\le \varepsilon$ for all $x\in M_i$.
If $\varepsilon\to0$, then
$Z_{i,\varepsilon}:=a_{i,\varepsilon}^{-1}(0)\to Z_{h_i}$.
This implies that for a sufficiently small $\varepsilon$ there is an equivariant bijection between $Z_{i,\varepsilon}$ and $Z_{h_i}$. Therefore, $\deg_G(a_{i,\varepsilon})=\deg_G(h_i)$.

Let $\varepsilon$ be  sufficiently small. Set $f_i:=a_{i,\varepsilon}$.
From Lemma \ref{lm4} it follows that there is an equivariant simplicial generic $F: W \to {\Bbb R}^n$ with $F|_{M_i}=f_i$ such that  $(Z_F,Z_{f_0},Z_{f_1})$ is a free $G$-cobordism.
 We have $m=n$.  Then $Z_F$ is a submanifold of $W$ of dimension one. Therefore $Z_F$ consists of arcs $\gamma_1,\ldots,\gamma_\ell$ with ends in $Z_{f_i}$ and cycles which do not intersect $M_i$.

Since any continuous map $s:\gamma_k\to \gamma_k$ has a fixed point,  $G$ cannot act free on  $\gamma_k$. Therefore, $G$ acts free on the set of all arcs $\{\gamma_k\}$. Moreover, the ends of any arc cannot lie in the same orbit of $G$.   From this it follows that  $|Z_{f_1}|=|Z_{f_0}|$ (mod 2$|G|)$, i.e. $\deg_G(f_0)=\deg_G(f_1)$.
\end{proof}

\begin{theorem} 
\label{thmdeg}	Let $G$ be a finite group acting linearly on ${\Bbb R}^n$ with $({\Bbb R}^n)^G=\{0\}$. Let $M^n$ be a closed connected PL free $G$-manifold. If there is a closed PL free $G$-manifold $N^n$ which is free $G$-cobordant to $M^n$ and a  continuous equivariant  transversal to zeros  $h: N^n\to {\Bbb R}^n$  with $\deg_G(h)=1$, then for any continuous equivariant  $f: M^n\to {\Bbb R}^n$ the  zero set $Z_f$ is not empty.
\end{theorem}
\begin{proof} Let $\deg_G(h)=1$.
 Suppose that $Z_f=\emptyset$.  Since $M$ is compact, there is an $\varepsilon>0$ such that
$||f(x)||\ge \varepsilon$ for all $x\in M$.
From Lemma \ref{lm3} it follows that there exists a generic $\tilde f$ such that $||f(x)-\tilde f(x)||\le \varepsilon/2$ for all $x\in M$.
Then $||\tilde f(x)||\ge \varepsilon/2$ for all $x\in M$ and $Z_{\tilde f}=\emptyset$. Therefore, $\deg_G(\tilde f)=0$. On the other hand, Lemma \ref{lm5}
implies $0=\deg_G(\tilde f)=\deg_G(h)=1$, a contradiction.
\end{proof}

\noindent{\bf Remark.} Actually, Lemma \ref{lm5} immediately implies that the assumption in the theorem: {\it There is a free $G$-manifold $N^n$ which is free $G$-cobordant to $M^n$ and a  continuous equivariant  transversal to zeros  $h: N^n\to {\Bbb R}^n$  with $\deg_G(h)=1$}  is equivalent to the following statement: {\it There is a  continuous equivariant  transversal to zeros map $h: M^n\to {\Bbb R}^n$  with $\deg_G(h)=1$.} However, since the assumption in the theorem is more general, it sometimes can be more easily checked for $N$ which are free $G$-cobordant to $M$. (For instance, see our proof of Theorem \ref{thmBUT}.)

\section{BUT manifolds}

 In this section, we consider the classical case $G={\Bbb Z}_2$. Let  $M$ be a closed PL-manifold  with a free simplicial involution $T:M\to M$, i.e. $T^2(x)=x$ and $T(x)\ne x$ for all $x\in M$. For any ${\Bbb Z}_2$-manifold $(M,T)$ we say that a map $f:M^m \to {\Bbb R}^n$  is {\it antipodal} (or equivariant) if $f(T(x))=-f(x)$.

\medskip

\noindent{\bf Definition.}
We say that a closed PL free ${\Bbb Z}_2$-manifold $(M^n,T)$ is a {\it BUT (Borsuk-Ulam Type) manifold} if for any continuous  $g:M^n \to {\Bbb R}^n$ there is a point $x\in M$ such that $g(T(x))=g(x)$. Equivalently, if a continuous  map $f:M^n \to {\Bbb R}^n$  is { antipodal},  then  $Z_f$ is not empty.

\medskip

Let us recall several  facts about ${\Bbb Z}_2$-cobordisms which are related to our main theorem in this section.
 We write ${\mathfrak N}_n$ for the group of unoriented cobordism classes of $n$-dimensional manifolds.    Thom's cobordism theorem says that
the graded ring of cobordism classes  ${\mathfrak N}_*$ is
${\Bbb Z}_2[x_2, x_4, x_5, x_6, x_8, x_9,\ldots]$ with one generator $x_k$ in each degree $k$ not of the form $2^i - 1$.
Note that $x_{2k}=[{\Bbb R}{\Bbb P}^{2k}]$.

Let  ${\mathfrak N}_*({\Bbb Z}_2)$ denote the unoriented cobordism group of  free involutions. Then ${\mathfrak N}_*({\Bbb Z}_2)$ is a free ${\mathfrak N}_*$-module with basis $[{\Bbb S}^n,A]$, $n\ge0$, where $[{\Bbb S}^n,A]$ is the cobordism class of the antipodal involution on the $n$-sphere \cite[Theorem 23.2]{CF}. Thus,  each ${\Bbb Z}_2$-manifold can be uniquely represented in ${\mathfrak N}_n({\Bbb Z}_2)$  in the form:
$$
[M,T]=\sum\limits_{k=0}^n {[V^k][{\Bbb S}^{n-k},A]}.
$$

\begin{theorem} \label{thmBUT} Let $M^n$ be a closed connected PL-manifold with a free simplicial involution $T$. Then the following statements are equivalent:

\noindent (a) $M$ is a BUT manifold.

\noindent (b)  $M$ admits an antipodal continuous transversal to zeros map $h:M^n \to {\Bbb R}^n$ with $\deg_{{\Bbb Z}_2}(h)=1$.

\noindent (c) $[M^n,T]=[{\Bbb S}^n,A]+[V^1][{\Bbb S}^{n-1},A]+\ldots+[V^n][{\Bbb S}^0,A]$ in ${\mathfrak N}_n({\Bbb Z}_2)$.
\end{theorem}

\begin{proof} {\bf (1)} First we prove that {\it (c) is equivalent to (b).} Let
$$[M,T]=[V^0][{\Bbb S}^n,A]+[V^1][{\Bbb S}^{n-1},A]+\ldots+[V^n][{\Bbb S}^0,A].$$
Denote by
$$N:=V^1\times{\Bbb S}^{n-1}\sqcup\ldots\sqcup V^n\times{\Bbb S}^{0}.$$

Note that for any $k$, where $0\le k<n$, the standard antipodal embedding of ${\Bbb S}^{k}$ into ${\Bbb R}^{n}$ has no zeros. This implies that there is a generic antipodal $p: N^n\to {\Bbb R}^n$  with $Z_p=\emptyset$. If $[V^0]=0$, then $M$ is free 
${\Bbb Z}_2$-cobordant to $N$. Therefore, Lemma \ref{lm5} yields that for any generic antipodal $f:M\to {\Bbb R}^n$  we have $\deg_{{\Bbb Z}_2}(f)=\deg_{{\Bbb Z}_2}(p)=0$. This contradicts (b), and thus $[V^0]=1\in {\Bbb Z}_2$.

On the other hand, if $[V^0]=1$, then $M$ is free ${\Bbb Z}_2$-cobordant to $L:={\Bbb S}^n\sqcup N$. Take any generic antipodal $h: M^n\to {\Bbb R}^n$. Since for any generic antipodal $g:L\to {\Bbb R}^n$ we have $\deg_{{\Bbb Z}_2}(g)=1$, it again follows from Lemma \ref{lm5} that $\deg_{{\Bbb Z}_2}(h)=1$.

\medskip

\noindent {\bf (2)} Theorem \ref{thmdeg} yields: {\it (b) implies (a).}

\medskip

\noindent {\bf (3)}  {\it (a) implies (c).} In fact, it follows from \cite[Theorem 3]{RS}.\footnote{I would to thank Alexey Volovikov who noted this theorem. He as well as Pavle Blagojevi\'c and Roman Karasev also sent me another proof of \cite[Theorem 3]{RS}.} 
This theorem yields that if $w_1^n(M/{\Bbb Z}_2)=0\in H^n(M/{\Bbb Z}_2,{\Bbb Z}_2)$, then there is an antipodal continuous map $h:M^n\to{\Bbb S}^{n-1}$, i.e. $Z_h=\emptyset.$ Therefore, $w_1^n(M/{\Bbb Z}_2)\ne0$. It holds if and only if we have (c). 
\end{proof}

\noindent{\bf Remark.} Actually, the list of equivalent versions in Theorem \ref{thmBUT} can be continued. {Tucker's lemma} is a discrete version of the Borsuk-Ulam theorem:\\
{\it Let $\Lambda$ be any equivariant triangulation  of ${\Bbb S}^n$. Let $$L:V(\Lambda)\to \{+1,-1,+2,-2,\ldots, +n,-n\}$$ be an equivariant (or Tucker) labelling, i.e.   $L(T(v))=-L(v)$). Then there exists a {complementary edge}
 in $\Lambda$ such that its two vertices are labelled by opposite numbers} (see \cite[Theorem 2.3.1]{Mat}).

 Let $M$ be as in Theorem \ref{thmBUT}. Using the same arguments as in \cite[Theorem 2.3.2]{Mat}, it can be proved that the following statement is equivalent to (a):

\medskip

\noindent {\it (d) For any equivariant labelling of an equivariant triangulation of $M$ there is a complementary edge.}

\medskip

In fact, any equivariant labelling $L$  defines a simplicial map $f_L:M\to {\Bbb R}^n$.
Indeed, let $e_1,\ldots, e_n$ be an orthonormal basis of ${\Bbb R}^n$. We define  $f_L:\Lambda\to {\Bbb R}^n$ for $v\in V(\Lambda)$ by $f_L(v)=e_i$ if $L(v)=i$ and  $f_L(v)=-e_i$ if $L(v)=-i$. In the paper \cite{Mus}, it is shown that $M$ is a BUT manifold if and only if

\medskip

\noindent {\it (e)  There exist an equivariant triangulation $\Lambda$ of $M$ and an equivariant labelling of $V(\Lambda)$  such that $f_L:\Lambda\to {\Bbb R}^n$ is transversal to zeros and the number of complementary edges is  $4k+2$, where $k=0,1,2,\ldots$}

\medskip

Lyusternik and Shnirelman proved in 1930 that for
any cover $F_1,\ldots,F_{n+1}$ of the sphere ${\Bbb S}^n$ by $n+1$ closed sets, there
is at least one set containing a pair of antipodal points (that is,
$F_i\cap(-F_{i})\ne\emptyset$). Equivalently, for any cover $U_1,\ldots,U_{n+1}$ of ${\Bbb S}^n$ by $n+1$ open sets, there is at least one set containing a pair of antipodal points \cite[Theorem 2.1.1]{Mat}. By the same arguments it can be shown that $M$ is a BUT manifold if and only if

\medskip

\noindent{\it (f) $M$ is a Lusternik-Shnirelman type manifold, i.e.  for
any cover $F_1,\ldots,F_{n+1}$ of  $M^n$ by $n+1$ closed (respectively,  by  $n+1$ open) sets, there is at least one set containing a pair $(x,T(x))$.}

\medskip

Denote by $\cat(X)$  the Lusternik-Shnirelman category of a space $X$, i.e. the smallest $m$ such that there exists an open covering $U_1,\ldots,U_{m+1}$ of $X$ with each $U_i$ contractible to a point in $X$. It is not hard to prove the last statement in this Remark:

\medskip

\noindent {\it (g)  $M$ is a BUT manifold if and only if $\cat(M/T)=\cat({\Bbb R}{\Bbb P}^n)=n$.}

\medskip

Usually, it is not easy to find  $\cat(X)$. Here, Theorems 2-4 yield interesting possibilities to find lower bounds for $\cat(M/G)$.

\medskip

\begin{cor} \label{corMcM} Let $M$ be any closed manifold. Then for the connected sum $M\# M$ there  exists a ``centrally symmetric'' free involution such that $M\# M$ is a BUT manifold.
\end{cor}
\begin{proof}
The Whitney embedding theorem states that any smooth or simplicial $n$-dimensional manifold can be embedded in Euclidean $2n$-space. Consider an embedding of $M$ in ${\Bbb R}^{2n}$ with coordinates $(x_1,\ldots,x_{2n})$. Let $M_-$ (respectively, $M_+$) denote the set of points in $M\subset {\Bbb R}^{2n}$ with $x_1<0$ (respectively, $x_1>0$). Let $M_0$ be the set of points in $M$ with $x_1=0$.
Without loss of generality we can assume that $M$ is embedded in ${\Bbb R}^{2n}$ in such a way that $M_-$ is homeomorphic to an open $n$-ball
and $M_0$ is  a sphere $x_1^2+\ldots+x_{n+1}^2=1$ with  $x_{n+2}=\ldots=x_{2n}=0$. Then the central symmetry $s$ $(s(x)=-x)$ is well defined on
$X:=s(M_+)\bigcup M_0\bigcup M_+$ as a free involution, and $X$ is homeomorphic to $M\# M$.  Consider the projection $h$ of $X$ onto the $n$-plane $x_{n+1}=\ldots=x_{2n}=0$, i.e. $h(x_1,\ldots, x_{2n})=(x_1,\ldots,x_n)$.  Since $Z_h\subset M_0$, we have $Z_h=Z_t$, where $t:=h|_{M_0}$. On the other hand, $t:M_0\to{\Bbb R}^n$ is an orthogonal projection of ${\Bbb S}^n$. Thus, $|Z_h|=2$  and Theorem \ref{thmBUT}(b) implies that $M\#M$ is a BUT manifold.
\end{proof}


\section{Obstructions  for $G$-maps in cobordisms}

In the previous sections, we considered equivariant maps $f:M^m \to {\Bbb R}^n$ with $m=n$. Now we extend this approach to the case $m\ge n$.


Consider closed PL manifolds with an $H$-structure (such as an orientation). One can define a ``cobordism with $H$-structure'', but there are various technicalities. In each particular case, cobordism is an equivalence relation on manifolds. A basic question is to determine the equivalence classes for this relationship, called the cobordism classes of manifolds. These form a graded ring called the cobordism ring $\Omega^H_*$, with grading by dimension, addition by disjoint union, and multiplication by cartesian product.

Let  ${\Omega}_*^H(G)$ denote the PL cobordism group with $H$-structure of  free simplicial actions of a finite group $G$. Let $\rho:G\to \gl(n,{\Bbb R})$ be a representation of a group $G$ on ${\Bbb R}^n$ which also has $H$-structure. Lemma \ref{lm4} shows  that for a generic simplicial equivariant  map  $f: M^m\to{\Bbb R}^n$ the cobordism class of the manifold $Z_f$ is uniquely defined up to cobordisms and so well defines a homomorphism
$$
\mu_{\rho}^G:\Omega_m^H(G)\to \Omega_{m-n}^H(G).
$$

\noindent{\bf Remark.} Note that the homomorphism $\mu_{\rho}^G:\Omega_m^H(G)\to \Omega_{m-n}^H(G)$ depends only on  a representation of a group $G$ on ${\Bbb R}^n$. For some groups, this homomorphism is known in algebraic topology. For instance, if $G={\Bbb Z}_2$ and $H=O$ (unoriented cobordisms, i.e. $\Omega_*^O({\Bbb Z}_2)={\mathfrak N}_*({\Bbb Z}_2)$), then $\mu=\Delta^k$, where $\Delta$ is called the {\it Smith homomorphism}. Conner and Floyd \cite[Theorem 26.1]{CF} give the following definition of $\Delta$: {\it Let $T$ be a free involution on a closed manifold $M$. For $n\ge m$ there exists an antipodal generic map $f:M^m\to{\Bbb S}^n$.
Let $X^{m-1}=f^{-1}({\Bbb S}^{n-1})$. Then the map $\Delta:\Omega_m^O({\Bbb Z}_2)\to \Omega_{m-1}^O({\Bbb Z}_2)$ which is defined by $[M^m,T]\to [X^{m-1},T|_{X}]$ is a homomorphism and $\Delta$ does not depend on $n$ and $f$.}


\medskip

The invariant $\mu_{\rho}^G$ is an obstruction for the existence of equivariant maps
$f: M\to {\Bbb R}^n\setminus \{0\}$. Namely, we have the following theorem.

\begin{theorem} \label{thmGcb} Let $M^m$ be a closed PL $G$-manifold with a free action $\tau$. Let $\rho$ be a  linear action of $G$ on ${\Bbb R}^n$. Let us assume that actions, manifolds, and maps  are with $H$-structure. Suppose that $\mu_{\rho}^G([M,\tau])\ne 0$ in $\Omega_{m-n}^H(G).$ Then for any continuous equivariant map $f: M^m\to {\Bbb R}^n$ the set of zeros $Z_f$ is not empty.
\end{theorem}
\begin{proof} Actually the proof of the theorem and of the corollary are almost word for word the same as the proof of Theorem \ref{thmdeg}.
Let us suppose that $f^{-1}(0)=\emptyset$.  Since $M$ is compact, there is an $\varepsilon>0$ such that
$||f(x)||\ge \varepsilon$ for all $x\in M$.
From Lemma 1 and Lemma \ref{lm3}, it follows that there exists a generic $h$ such that $||f(x)-h(x)||\le \varepsilon/2$ for all $x\in M$. By Lemma \ref{lm4} we have
$\mu_{\rho}^G([M,\tau])=[Z_h]_G\ne 0$ in $\Omega_{m-n}^H(G)$. Then $Z_h$ is a submanifold of $M$ of dimension $m-n$.

Since $Z_h\ne\emptyset$,  there is $x\in M$ such that $h(x)=0$.
Thus, $\varepsilon/2\ge ||f(x)-h(x)||=||f(x)||\ge \varepsilon>0$, a contradiction.
\end{proof}

\begin{cor} \label{cor2}  Let manifolds, actions, and maps be as above.  Suppose that there is a  continuous equivariant transversal (in zeros) map $h: M^m\to {\Bbb R}^n$ such that the set $Z_h$ is a closed submanifold in $M$ of codimension $n$ and $[Z_h]_G\ne 0$ in $\Omega_{m-n}^H(G).$ Then for any continuous equivariant map $f: M\to {\Bbb R}^n$ the set of zeros $Z_f$ is not empty.
\end{cor}
\begin{proof}
By Lemma \ref{lm3}, for any $\varepsilon>0$ there is a generic map $a_\varepsilon$ such that $||a_\varepsilon(x)-h(x)||\le \varepsilon$ for all $x\in M$. If $\varepsilon\to0$, then $Z_\varepsilon:=a_\varepsilon^{-1}(0)\to Z$. This implies that for a sufficiently small $\varepsilon$, there is a homeomorphism between $Z_\varepsilon$ and $Z_h$. Therefore, $\mu_{\rho}^G([M,\tau])=[Z_h]_G\ne 0$ in $\Omega_{m-n}^H(G).$
\end{proof}




Let  $G={\Bbb Z}_2$ and $\rho(x)=-x$. Denote by $\mu_n:=\mu_{\rho}^{{\Bbb Z}_2}$.
Then we have
$$\mu_n:{\mathfrak N}_{m}({\Bbb Z}_2)\to{\mathfrak N}_{m-n}({\Bbb Z}_2).$$

If $u\in{\mathfrak N}_m({\Bbb Z}_2)$ is given explicitly, then we can easily find $\mu_n(u)$.
\begin{lemma} \label{lm6}
$$
\mu_n([M^m,T])=\mu_n\left(\sum\limits_{k=0}^m {[V^k]\,[{\Bbb S}^{m-k},A]}\right)= \sum\limits_{k=0}^{m-n} {[V^k]\,[{\Bbb S}^{m-n-k},A]}.
$$
\end{lemma}
\begin{proof} Since $Z_{f_i}=F^{-1}(0)\bigcap M_i$ in Lemma \ref{lm4} are cobordant to each other, we can consider the standard projection $Pr:{\Bbb S}^{d}\to {\Bbb R}^{n}$. Then $Pr^{-1}(0)={\Bbb S}^{d-n}$. This yields  $\mu_n\left([{\Bbb S}^d,A]\right)=[{\Bbb S}^{d-n},A]$.
It is clear that the lemma follows from this equality.
\end{proof}

Lemma \ref{lm6} and Theorem \ref{thmGcb} imply the following corollary.
\begin{cor} Let
$
[M^m,T]={[V^{m-n}][{\Bbb S}^{n},A]}+{[V^{m-n+1}][{\Bbb S}^{n-1},A]}+\ldots+{[V^{m}][{\Bbb S}^{0},A]}
$ in ${\mathfrak N}_m({\Bbb Z}_2)$
with $[V^{m-n}]\ne0$ in ${\mathfrak N}_{m-n}$.  Then for any antipodal continuous  map $f:M^m\to {\Bbb R}^n$  the set $Z_f$ is not empty.
\end{cor}

This corollary can be extended for the case $G=({\Bbb Z}_2)^k={\Bbb Z}_2\times\ldots\times{\Bbb Z}_2$. Note that  ${\mathfrak N}_*(({\Bbb Z}_2)^k)={\mathfrak N}_*({\Bbb Z}_2)\otimes\ldots\otimes{\mathfrak N}_*({\Bbb Z}_2)$  \cite[Sec. 29]{CF}. In other words,
 ${\mathfrak N}_*(({\Bbb Z}_2)^k)$ is a free ${\mathfrak N}_*$-module with generators $\{\gamma_{i_1}\otimes\ldots\otimes{\gamma_{i_k}}\}, \, i_1,\ldots,i_k=0,1,\ldots$, where $\gamma_i:=[{\Bbb S}^{i},A]\in {\mathfrak N}_i({\Bbb Z}_2)$.

\begin{cor} \label{cor4} Let $M$ be a closed PL manifold with a free action $\Psi$ of $({\Bbb Z}_2)^k$. Let
$
[M^m,\Psi]=\sum{[V^{i_1,\ldots,i_k}]\,\gamma_{i_1}\otimes\ldots\otimes\gamma_{i_k}}$ in ${\mathfrak N}_m(({\Bbb Z}_2)^k)$
with $[V^{i_1,\ldots,i_k}]\ne0$ in ${\mathfrak N}_{*}$. Consider $({\Bbb Z}_2)^k$ with generators $T_1,\ldots,T_k$ acting on ${\Bbb R}^{i_1}\oplus\ldots\oplus{\Bbb R}^{i_k}$ by
$T_\ell(x)=-x$ for $x\in {\Bbb R}^{i_\ell}$.
 Then for any  continuous equivariant map
$f:M^m\to {\Bbb R}^{i_1}\oplus\ldots\oplus{\Bbb R}^{i_k}$  the set $Z_f$ is not empty.
\end{cor}

 Let $\tilde G(4,2)$ be the oriented Grassmann manifold, that is, the space of all oriented 2-dimensional subspaces of ${\Bbb R}^4$. Note that ${\Bbb Z}_2\times{\Bbb Z}_2$ acts on $\tilde G(4,2)$, where for $x\in\tilde G(4,2)$, $T_1(x)$ changes the orientation of $x$ and $T_2(x)$ is the oriented orthogonal complement of an oriented 2-plane $x$ in  ${\Bbb R}^4$.

 It is well known that $\tilde G(4,2)$ is diffeomorphic to ${\Bbb S}^{2}\times{\Bbb S}^{2}$ \cite[3.2.3]{FR}. For $(u,v)\in{\Bbb S}^{2}\times{\Bbb S}^{2}$, we have $T_1(u,v)=(-u,-v)$ and $T_2(u,v)=(u,-v)$. Then for $\Psi=(T_1,T_2)$ we have: $[\tilde G(4,2),\Psi]=\gamma_2\otimes\gamma_2$ in ${\mathfrak N}_4({\Bbb Z}_2\times{\Bbb Z}_2)$. Then Corollary \ref{cor4} (or Theorem \ref{thmdeg}) yields:
 \begin{cor} Let ${\Bbb Z}_2\times{\Bbb Z}_2$ act on  ${\Bbb R}^4={\Bbb R}^2\oplus{\Bbb R}^2$ by $T_1(u,v)=(-u,-v)$ and $T_2(u,v)=(u,-v)$. Then for any equivariant continuous $f:\tilde G(4,2)\to {\Bbb R}^4$ the set $Z_f$ is not empty. \end{cor}

\medskip

\noindent{\bf Acknowledgment.} I  wish to thank Arseniy Akopyan, Imre  B\'ar\'any, Pavle Blagojevi\'c, Mike Field, Ji\v{r}\'{i} Matou\v{s}ek,  Alexey Volovikov, and G\"unter Ziegler for helpful discussions and  comments. I am most grateful to Roman Karasev and to the anonymous referee for several critical comments and corrections.

 \medskip

O. R. Musin, Department of Mathematics, University of Texas at Brownsville, 80 Fort Brown, Brownsville, TX, 78520.

 {\it E-mail address:} oleg.musin@utb.edu

\end{document}